\RequirePackage{fix-cm}
\documentclass[smallcondensed,envcountsame]{svjour3} 
\smartqed  
\usepackage{graphicx, amssymb}
\usepackage{times,a4wide,mathrsfs}
\usepackage{amsmath}
\usepackage{amsfonts}

\newcommand{\C}{\mathbb{C}}
\newcommand{\ZZ}{\mathbb{Z}}

\newcommand{\QQ}{\mathbb{Q}}
\newcommand{\NN}{\mathbb{N}}

\newcommand{\OO}{\mathcal O}

\newcommand{\MM}{\mathcal M}

\DeclareMathOperator{\gr}{Gr}

\DeclareMathOperator{\ima}{Im}
\newcommand{\rom}{\romannumeral}

\newcommand\undermat[2]{
  \makebox[0pt][l]{$\smash{\underbrace{\phantom{
    \begin{matrix}#2\end{matrix}}}_{\text{$#1$}}}$}#2}

\newtheorem{convention}{Conventions}

\newtheorem{nonumbering}{Theorem}

\newtheorem{nonumberingc}{Corollary}

 \journalname{Abh. Math. Semin. Univ. Hambg.}

\begin{document}

\title{Hard Lefschetz for Chow groups of generalized Kummer varieties}

\author{Robert Laterveer}

\institute{CNRS - IRMA, Universit\'e de Strasbourg \at
              7 rue Ren\'e Descartes \\
              67084 Strasbourg cedex\\
              France\\
              \email{{\tt laterv@math.unistra.fr}}   }

\date{Received: date / Accepted: date}
%\date{}

\maketitle

\begin{abstract} The main result of this note is a hard Lefschetz theorem for the Chow groups of generalized Kummer varieties. The same argument also proves hard Lefschetz for Chow groups of Hilbert schemes of abelian surfaces. As a consequence, we obtain new information about certain pieces of the Chow groups of generalized Kummer varieties, and Hilbert schemes of abelian surfaces.
The proofs are based on work of Shen--Vial and Fu--Tian--Vial on multiplicative Chow--K\"unneth decompositions. 
\end{abstract}



\section{Introduction}

The object of this note is to prove the statement announced in its title.
Let $X$ be a smooth projective variety over $\C$.
We will be interested in the Chow groups $A^i(X):=CH^i(X)_{\QQ}$ (of codimension $i$ cycles on $X$ with rational coefficients modulo rational equivalence).

Thanks to work of Beauville \cite{Beau}, 
%(and Deninger--Murre \cite{DM}), 
the Chow group of an abelian variety $A$ splits into pieces
  \[ A^i(A)=\bigoplus_j A^i_{(j)}(A) \ ,\]
  where $A^i_{(j)}(A):=\{ a\in A^i(A) \ \vert\ [n]^\ast(a)=n^{2i-j}a\ \forall n\in\ZZ\}$ (here $[n]\colon A\to A$ denotes the multiplication by $n$ morphism).
   This makes the Chow ring $A^\ast(A)$ into a bigraded ring $A^\ast_{(\ast)}(A)$. It is known \cite{DM} this splitting is induced by a Chow--K\"unneth decomposition (in the sense of Definition \ref{ck}). 
   The piece $A^i_{(j)}(A)$ is expected to be isomorphic to the graded
  $\gr^j_F A^i(A)$ of the conjectural Bloch--Beilinson filtration.
  
  Thanks to work of K\"unnemann, there is a ``hard Lefschetz'' result for these pieces of the Chow groups of abelian varieties:
  
  \begin{theorem}[K\"unnemann \cite{Kun}]\label{kun} Let $A$ be an abelian variety of dimension $g$. Let $h\in A^1(A)$ be a symmetric ample class. Then there are isomorphisms
   \[   \cdot h^{g-2i+j}\colon\ \  A^i_{(j)}(X)\ \xrightarrow{\cong}\ A^{g-i+j}_{(j)}(X)\ \ \ \hbox{for\ all\ } 0\le 2i-j\le g\ .\]  
   \end{theorem}
  
  A similar bigraded ring structure $A^\ast_{(\ast)}$ is expected to exist on the Chow ring of any hyperk\"ahler variety (i.e., a projective irreducible holomorphic symplectic manifold, cf. \cite{Beau0}, \cite{Beau1}). This expectation is related to Beauville's conjectural ``splitting property'' \cite{Beau3}. One series of hyperk\"ahler varieties where this has recently been verified is that of generalized Kummer varieties. Indeed, Fu--Tian--Vial \cite{FTV} prove that any generalized Kummer variety $X$ has a multiplicative Chow--K\"unneth decomposition (in the sense of \cite{SV}, cf. Definition \ref{mck}); this implies the Chow ring $A^\ast(X)$ has a bigraded ring structure $A^\ast_{(\ast)}(X)$. 
  
 It seems natural to ask whether these pieces $A^i_{(j)}(X)$ satisfy a hard Lefschetz theorem similar to Theorem \ref{kun}. The main result of this note answers this question affirmatively:
  
  \begin{nonumbering}[=theorem \ref{main3}] Let $X=K_m(A)$ be a generalized Kummer variety of dimension $2m$. Let $h\in A^1(X)$ be ample. Then intersection induces isomorphisms
     \[ \cdot h^{2m-2i+j}\colon\ \ A^i_{(j)}(X)\ \xrightarrow{\cong}\ A^{2m-i+j}_{(j)}(X)\ \ \ \hbox{for\ all\ } 0\le 2i-j\le 2m\ .\]
   \end{nonumbering}
  
  Given that generalized Kummer varieties are ``built out of'' abelian varieties, one might expect Theorem \ref{main3} is proven by reducing to Theorem \ref{kun}.
  This is probably feasible, yet this is {\em not\/} the way we proceed. Instead, we prefer to start from scratch and present a general statement (Theorem \ref{main0}), which applies at once to abelian varieties, to Hilbert schemes of abelian surfaces, and to generalized Kummer varieties.

  As a consequence of Theorem \ref{main3}, we find that certain pieces $A^i_{(j)}(X)$ vanish:
  
  \begin{nonumberingc}[=Corollary \ref{cor3}] Let $X=K_m(A)$ be a generalized Kummer variety. Then
    \[ A^i_{(i)}(X)=0\ \ \ \hbox{for\ all\ }i\ \hbox{odd}\ .\]
   \end{nonumberingc}
   
   We can also say something about the ``extremal'' pieces:
   
    \begin{nonumberingc}[=Corollaries \ref{cor5} and \ref{cor6}] Let $X$ be a generalized Kummer variety, or the Hilbert scheme $A^{[m]}$ of an abelian surface $A$. Then
    \[ A^i_{(i)}(X)\ \subset\ A^i_{alg}(X)\  , \ \ \ \  \ \    A^{i+1}_{(i)}(X)\ \subset\ A^i_{\otimes}(X) \ \ \ \hbox{for\ all\ }i>0\ .\]
    (Here, $A^i_{alg}$ and $A^i_\otimes$ denote the subgroup of algebraically trivial, rep. smash--nilpotent cycles.)
   \end{nonumberingc}   
   
   It would be interesting to extend Theorem \ref{kun} (and the corollaries) to other varieties that have a multiplicative 
   Chow--K\"unneth decomposition.

\begin{convention} In this article, the word {\sl variety\/} will refer to a reduced irreducible scheme of finite type over $\C$. A {\sl subvariety\/} is a (possibly reducible) reduced subscheme which is equidimensional. 

{\bf All Chow groups will be with rational coefficients}: we will denote by $A_j(X)$ the Chow group of $j$--dimensional cycles on $X$ with $\QQ$--coefficients; for $X$ smooth of dimension $n$ the notations $A_j(X)$ and $A^{n-j}(X)$ are used interchangeably. 

%We will write $B_j(X)$ for the group of $j$--dimensional cycles modulo algebraic equivalence with $\QQ$--coefficients; for $X$ smooth of dimension $n$ the notations $B_j(X)$ and $B^{n-j}(X)$ are used interchangeably. 

The notations 
  \[ A^j_{hom}(X)\ ,\ A^j_{AJ}(X)\ ,\  A^j_{alg}(X)\ ,\ A^j_\otimes(X) \] 
  will be used to indicate the subgroups of homologically trivial, resp. Abel--Jacobi trivial, resp. algebraically trivial, resp. smash-nilpotent cycles.
For a morphism $f\colon X\to Y$, we will write $\Gamma_f\in A_\ast(X\times Y)$ for the graph of $f$.
The contravariant category of Chow motives (i.e., pure motives with respect to rational equivalence as in \cite{Sc}, \cite{MNP}) will be denoted $\MM_{\rm rat}$.

%To avoid heavy notation, if $\tau\colon Y\to X$ is a closed inclusion and $a\in A_iY$, we will frequently write $a\in A_iX$ to indicate the proper push--forward $\tau_\ast(a)$. Likewise, for any inclusion $Y\subset X$ and $b\in A^jX$ we will often write
%  \[   b\vert_{Y}\ \ \in A^jY\]
 % to indicate the cycle class $\tau^\ast(b)$.

%The Griffiths group $\grif^j$ is the group of codimension $j$ cycles that are homologically trivial modulo algebraic equivalence, again with $\QQ$--coefficients. 

We will write $H^j(X)$ 
to indicate singular cohomology $H^j(X,\QQ)$.

%Given a group $G\subset\aut(X)$ of automorphisms of $X$, we will write $A^j(X)^G$ (and $H^j(X)^G$) for the subgroup of $A^j(X)$ (resp. $H^j(X)$) invariant under $G$.
\end{convention}

\section{Preliminary}

\subsection{Finite--dimensional motives}

We refer to \cite{Kim}, \cite{An}, \cite{Iv}, \cite{J4}, \cite{MNP} for the definition of finite--dimensional motive. 
An essential property of varieties with finite--dimensional motive is embodied by the nilpotence theorem:

\begin{theorem}[Kimura \cite{Kim}]\label{nilp} Let $X$ be a smooth projective variety of dimension $n$ with finite--dimensional motive. Let $\Gamma\in A^n(X\times X)_{\QQ}$ be a correspondence which is numerically trivial. Then there is $N\in\NN$ such that
     \[ \Gamma^{\circ N}=0\ \ \ \ \in A^n(X\times X)_{}\ .\]
\end{theorem}

 Actually, the nilpotence property (for all powers of $X$) could serve as an alternative definition of finite--dimensional motive, as shown by a result of Jannsen \cite[Corollary 3.9]{J4}.
Conjecturally, any variety has finite--dimensional motive \cite{Kim}. We are still far from knowing this, but at least there are quite a few non--trivial examples.

\subsection{MCK decomposition}

\begin{definition}[Murre \cite{Mur}]\label{ck} Let $X$ be a smooth projective variety of dimension $n$. We say that $X$ has a {\em CK decomposition\/} if there exists a decomposition of the diagonal
   \[ \Delta_X= \pi_0+ \pi_1+\cdots +\pi_{2n}\ \ \ \hbox{in}\ A^n(X\times X)\ ,\]
  such that the $\pi_i$ are mutually orthogonal idempotents and $(\pi_i)_\ast H^\ast(X)= H^i(X)$.
  
  (NB: ``CK decomposition'' is short--hand for ``Chow--K\"unneth decomposition''.)
\end{definition}

\begin{remark} The existence of a CK decomposition for any smooth projective variety is part of Murre's conjectures \cite{Mur}, \cite{J2}. 
%If a quotient variety $X$
%has finite--dimensional motive, and the K\"unneth components are algebraic, then $X$ has a CK decomposition (this can be proven just as \cite{J2}, where this is stated for smooth $X$).
\end{remark}

\begin{definition}[Shen--Vial \cite{SV}]\label{mck} Let $X$ be a smooth projective variety of dimension $n$. Let $\Delta_X^{sm}\in A^{2n}(X\times X\times X)$ be the class of the small diagonal
  \[ \Delta_X^{sm}:=\bigl\{ (x,x,x)\ \vert\ x\in X\bigr\}\ \subset\ X\times X\times X\ .\]
  An MCK decomposition is a CK decomposition $\{\pi_i\}$ of $X$ that is {\em multiplicative\/}, i.e. it satisfies
  \[ \pi_k\circ \Delta_X^{sm}\circ (\pi_i\times \pi_j)=0\ \ \ \hbox{in}\ A^{2n}(X\times X\times X)\ \ \ \hbox{for\ all\ }i+j\not=k\ .\]
  
 (NB: ``MCK decomposition'' is short--hand for ``multiplicative Chow--K\"unneth decomposition''.) 
  \end{definition}
  
  \begin{remark} The small diagonal (seen as a correspondence from $X\times X$ to $X$) induces the {\em multiplication morphism\/}
    \[ \Delta_X^{sm}\colon\ \  h(X)\otimes h(X)\ \to\ h(X)\ \ \ \hbox{in}\ \MM_{\rm rat}\ .\]
 Suppose $X$ has a CK decomposition
  \[ h(X)=\bigoplus_{i=0}^{2n} h^i(X)\ \ \ \hbox{in}\ \MM_{\rm rat}\ .\]
  By definition, this decomposition is multiplicative if for any $i,j$ the composition
  \[ h^i(X)\otimes h^j(X)\ \to\ h(X)\otimes h(X)\ \xrightarrow{\Delta_X^{sm}}\ h(X)\ \ \ \hbox{in}\ \MM_{\rm rat}\]
  factors through $h^{i+j}(X)$.
  It follows that if $X$ has an MCK decomposition, then setting
    \[ A^i_{(j)}(X):= (\pi^X_{2i-j})_\ast A^i(X) \ ,\]
    one obtains a bigraded ring structure on the Chow ring: that is, the intersection product sends $A^i_{(j)}(X)\otimes A^{i^\prime}_{(j^\prime)}(X) $ to  $A^{i+i^\prime}_{(j+j^\prime)}(X)$.
    It is expected that for any $X$ with an MCK decomposition, one has
    \[ A^i_{(j)}(X)\stackrel{??}{=}0\ \ \ \hbox{for}\ j<0\ ,\ \ \ A^i_{(0)}(X)\cap A^i_{hom}(X)\stackrel{??}{=}0\ ;\]
    this is related to Murre's conjectures B and D \cite{Mur}.

  The property of having an MCK decomposition is severely restrictive, and is closely related to Beauville's ``(weak) splitting property'' \cite{Beau3}. For more ample discussion, and examples of varieties with an MCK decomposition, we refer to \cite[Chapter 8]{SV}, as well as \cite{V6}, \cite{SV2}, \cite{FTV}.
    \end{remark}

In this note, we will rely on the following results:

\begin{theorem}[Vial \cite{V6}]\label{hilbk} Let $S$ be an abelian surface, and let $X=S^{[m]}$ be the Hilbert scheme of length $m$ subschemes of $S$. Then $X$ has an MCK decomposition.
\end{theorem}

\begin{proof} This is (part of) \cite[Theorem 1]{V6}. It is also reproven (from a slightly different viewpoint) as case (A) of \cite[Theorem 7.9]{FTV}.
  \end{proof}

\begin{theorem}[Fu--Tian--Vial \cite{FTV}]\label{ftv} Let $X$ be a generalized Kummer variety. Then $X$ has an MCK decomposition.
\end{theorem}

\begin{proof} This is case (B) of \cite[Theorem 7.9]{FTV}.
\end{proof}

\section{A general result}

\begin{theorem}\label{main0} Let $X$ be a smooth projective variety of dimension $g$, and let $h\in A^1(X)$ be an ample class. Assume the following:

\noindent
(\rom1) $X$ has an MCK decomposition $\{\pi_i^X\}$;

\noindent
(\rom2) $X$ has finite--dimensional motive;

\noindent
(\rom3) the standard Lefschetz conjecture $B(X)$ holds;

\noindent
(\rom4) the class $h$ is in $A^1_{(0)}(X)$.

Then there are isomorphisms
  \[ \cdot h^{g-2i+j}\colon\ \ A^i_{(j)}(X)\ \xrightarrow{\cong}\ A^{g-i+j}_{(j)}(X)\ \ \ \hbox{for\ all\ } 0\le 2i-j\le g\ .\]
  \end{theorem}
  
 \begin{proof} The cohomological hard Lefschetz theorem states that cup--product induces an isomorphism
   \[ L^{g-2i+j}\colon\ \ H^{2i-j}(X)\ \xrightarrow{\cong}\ H^{2g-2i+j}(X)\ ,\ \ \  0\le 2i-j\le g\ .\]
 By assumption (\rom3), there exists a correspondence $C^{g-2i+j}\in A^{2i-j}(X\times X)$ inducing the inverse to $L^{g-2i+j}$, i.e.
  \begin{equation}\label{isoh} \begin{split}   &C^{g-2i+j}\circ \pi^X_{2g-2i+j}\circ L^{g-2i+j}\circ \pi^X_{2i-j} = \pi^X_{2i-j}\ \ \ \hbox{in}\ H^{2g}(X\times X)\ ,\\
                                                                      &L^{g-2i+j}\circ \pi^X_{2i-j} \circ  C^{g-2i+j}\circ \pi^X_{2g-2i+j} = \pi^X_{2g-2i+j} \ \ \ \hbox{in}\ H^{2g}(X\times X)\ .\\
                                                                \end{split}\end{equation}      
                                                                      
  As is well--known, this implies there is an isomorphism of homological motives
   \[  L^{g-2i+j}\colon\ \ \ (X,\pi^X_{2i-j},0))\ \xrightarrow{\cong}\ (X,\pi^X_{2g-2i+j}, g-2i+j)\ \ \ \hbox{in}\ \MM_{\rm hom}\ ,\ \ \  0\le 2i-j\le g\ \]
   (with inverse given by $C^{g-2i+j}$).
   Assumption (\rom2) allows to upgrade to an isomorphism of Chow motives
   \begin{equation}\label{ch}  L^{g-2i+j}\colon\ \ \ (X,\pi^X_{2i-j},0))\ \xrightarrow{\cong}\ (X,\pi^X_{2g-2i+j}, g-2i+j)\ \ \ \hbox{in}\ \MM_{\rm rat}\ ,\ \ \  0\le 2i-j\le g\ \end{equation} 
    (with inverse given by $C^{g-2i+j}$).    
   Taking Chow groups of the motives on both sides of (\ref{ch}), we find an isomorphism
   \[ (\pi^X_{2g-2i+j}\circ L^{g-2i+j}\circ \pi^X_{2i-j})_\ast\colon \    A^i_{(j)}(X)\ \xrightarrow{\cong}\ A^{g-i+j}_{(j)}(X)\ .\]
   
   It remains to note that under assumption (\rom4), we have
   \[     ( L^{g-2i+j} )_\ast  =    (\pi^X_{2g-2i+j}\circ L^{g-2i+j})_\ast =    (\pi^X_{2g-2i+j}\circ L^{g-2i+j}\circ \pi^X_{2i-j})_\ast\colon \   A^i_{(j)}(X)\ \xrightarrow{}\ A^{g-i+j}_{}(X)\ ,\]
   since (because of the bigraded ring structure) we have an inclusion 
     \[ ( L^{g-2i+j} )_\ast A^i_{(j)}(X) =  \ima\Bigl( A^i_{(j)}(X)\xrightarrow{\cdot h^{g-2i+j}} A^{g-i+j}(X)\Bigr) \ \subset \ A^{g-i+j}_{(j)}(X)=(\pi^X_{2g-2i+j})_\ast A^{g-i+j}(X)\ .\]   
    In conclusion, we have proven there is an isomorphism
    \[ \cdot h^{g-2i+j}=   ( L^{g-2i+j} )_\ast\colon\ \ \ A^i_{(j)}(X)\ \xrightarrow{\cong}\ A^{g-i+j}_{(j)}(X)\ , \ \ \  0\le 2i-j\le g \ ,\]
    with inverse given by $(\pi^X_{2i-j}\circ C^{g-2i+j})_\ast$.
                
        \end{proof}

\section{Abelian varieties}

\begin{theorem}[K\"unnemann \cite{Kun}]\label{main1} Let $A$ be an abelian variety of dimension $g$. Let $h\in A^1(A)$ be an ample symmetric class. 
%and let $L$ denote the Lefschetz operator associated to $h$.
 Then there are isomorphisms
    \[ \cdot h^{g-2i+j}\colon\ \ A^i_{(j)}(X)\ \xrightarrow{\cong}\ A^{g-i+j}_{(j)}(X)\ \ \ \hbox{for\ all\ } 0\le 2i-j\le g\ .\]
  \end{theorem}

\begin{proof} We will check the assumptions of Theorem \ref{main0} are met with. For assumptions (\rom2) and (\rom3), this is well--known (\cite{Kim} resp. \cite{K}, \cite{K2}). Assumption (\rom1) is handled in \cite[Example 8.3]{SV} and \cite[Example 8.5]{SV}.      
 \end{proof}

\begin{remark} It is worth pointing out that in proving Theorem \ref{main1}, we have not really used the ampleness of the divisor class $h$. The only property we have used is that cupping with powers of $h$ induces isomorphisms in cohomology
  \[ \cup h^{g-2i+j}\colon\ \ \ H^{2i-j}(A)\ \xrightarrow{\cong}\ H^{2g-2i+j}(A)\ ,\ \ \  0\le 2i-j\le g\ .\]
  It follows that Theorem \ref{main1} is true more generally for {\em lef line bundles\/}, in the sense of \cite{CM2} (by definition, a line bundle is lef when it is the pullback of an ample line bundle under a semismall morphism).
  
  Also, we obtain the following ``mixed version'' of Theorem \ref{main1}. Let $h_1,\ldots,h_{g-2i+j}\in A^1(A)$ be symmetric ample classes. Then, since there is an isomorphism
  \[ \cup h_1\cup\cdots \cup h_{g-2i+j}\colon\ \ \  H^{2i-j}(A)\ \xrightarrow{\cong}\ H^{2g-2i+j}(A)\ \]
  \cite[Theorem 1.3]{Cat}, the above argument implies there is also an isomorphism
  \[\cdot h_1\cdot \ldots \cdot h_{g-2i+j}\colon\ \ \ A^i_{(j)}(A)\ \xrightarrow{\cong}\ A_{(j)}^{g-i+j}(A)\ .\]  
  \end{remark}

\section{Hilbert schemes of abelian surfaces}

\begin{theorem}\label{main2} Let $A$ be an abelian surface, and let $X=A^{[m]}$ be the Hilbert scheme of length $m$ subschemes of $A$. Let $h\in A^1_{(0)}(X)$ be an ample class. Then there are isomorphisms
  \[  \cdot h^{2m-2i+j}\colon\ \ A^i_{(j)}(X)\ \xrightarrow{\cong}\ A^{2m-i+j}_{(j)}(X)\ \ \ \hbox{for\ all\ } 0\le 2i-j\le 2m\ .\]
\end{theorem}

\begin{proof} This is proven by appealing to Theorem \ref{main0}. The first assumption of Theorem \ref{main0} is satisfied by virtue of Vial's result \cite{V6} (cf. Theorem \ref{hilbk} above); assumptions (\rom2) and (\rom3) follow from \cite{CM}. 
  \end{proof}

\section{Generalized Kummer varieties}

 \begin{definition} Let $A$ be an abelian surface. For any $m\in\NN$, let 
  \[    \pi\colon A^{[m+1]}\ \to\ A^{(m+1)}\]
  denote the Hilbert--Chow morphism from the Hilbert scheme $A^{[m+1]}$ to the symmetric product $A^{(m+1)}$. Let $\sigma\colon A^{(m+1)}\to A$ denote the addition morphism. Consider the composition
  \[ s\colon A^{[m+1]}\ \xrightarrow{\pi}\ A^{(m+1)}\ \xrightarrow{\sigma}\ A\ .\]
  The generalized Kummer variety is defined as the fibre
  \[  K_m(A):=s^{-1}(0)\ .\]
The variety $K_m(A)$ is a hyperk\"ahler variety of dimension $2m$ \cite{Beau1}.
\end{definition}

\begin{theorem}\label{main3} Let $X=K_m(A)$ be a generalized Kummer variety. Let $h\in A^1_{}(X)$ be an ample class. Then there are isomorphisms
  \[  \cdot h^{2m-2i+j}\colon\ \ A^i_{(j)}(X)\ \xrightarrow{\cong}\ A^{2m-i+j}_{(j)}(X)\ \ \ \hbox{for\ all\ } 0\le 2i-j\le 2m\ .\]
\end{theorem}

\begin{proof} The first three assumptions of Theorem \ref{main0} are satisfied (for assumption (\rom1) this is Fu--Tian--Vial's result (Theorem \ref{ftv} above), assumptions (\rom2) and (\rom3) follow from \cite{FTV} or \cite{Xu}). Assumption (\rom4) is immediate, because $A^1(X)=A^1_{(0)}(X)$.
  \end{proof}

  \section{Some corollaries}

   \begin{corollary}\label{cor4} Let $X$ be a hyperk\"ahler variety birational to a generalized Kummer variety $X^\prime$, and let $h\in A^1(X)$ be an ample class. Then there are isomorphisms
     \[  \cdot h^{2m-2i+j}\colon\ \ A^i_{(j)}(X)\ \xrightarrow{\cong}\ A^{2m-i+j}_{(j)}(X)\ \ \ \hbox{for\ all\ } 0\le 2i-j\le 2m\ .\]
\end{corollary}

\begin{proof} The point is that thanks to work of Rie\ss\, \cite{Rie}, there is an isomorphism of graded commutative $\QQ$--algebras
  \[ \phi_\ast\ \colon\ \ \ A^\ast(X)\ \xrightarrow{\cong}\ A^\ast(X^\prime)\ .\]
This implies that
$X$ still enters into the set--up of Theorem \ref{main0}. 
    \end{proof}

\begin{corollary}\label{cor3}  Let $X=K_m(A)$ be a generalized Kummer variety. Then
  \[ A^i_{(i)}(X)=0\ \ \ \hbox{for\ all\ } i\ \hbox{odd}\ .\]
  \end{corollary}
  
\begin{proof} Lin \cite{Lin} has proven that
  \[ A^{2m}_{(i)}(X) =0\ \ \ \hbox{for\ all\ } i\ \hbox{odd}\ .\]
  Applying Theorem \ref{main3}, we find there are isomorphisms
  \[ A^i_{(i)}(X)\ \xrightarrow{\cong}\ A^{2m}_{(i)}(X)\ \ \ \hbox{for\ all\ }i\ .\]
  This proves the corollary.
 \end{proof}  

  \begin{remark} Corollary \ref{cor3} can be seen as a ``motivic'' manifestation of the fact that $H^i(X,\OO_X)=0$ for $i$ odd (as such, Corollary \ref{cor3} should hold for all hyperk\"ahler varieties).
  
  On a side note, we mention that applying the Bloch--Srinivas method to the above--mentioned result of Lin, one can actually prove
  \[ H^i(X) = N^1 H^i(X)\ \ \ \forall\ i\ \hbox{odd} \]
  (where $N^\ast$ denotes the coniveau filtration \cite{BO}).
  \end{remark}

   \begin{corollary}\label{cor5} Let $X$ be an abelian variety, or the Hilbert scheme $A^{[m]}$ of an abelian surface, or a generalized Kummer variety. Then
   \[  A^i_{(i)}(X)\ \subset\ A^i_{alg}(X)\ \ \ \forall i>0\ .\]
   \end{corollary}
   
   \begin{proof} Let $n:=\dim X$. The proof of Theorem \ref{main0} implies there exists a correspondence $C$ such that
   \[ A^i_{(i)}(X)\ \xrightarrow{ \cdot h^{n-i}}\ A^{n}_{(i)}(X)\ \xrightarrow{C_\ast}\ A^i_{(i)}(X) \]
   is the identity. But for $i>0$, any cycle $b\in A^{n}_{(i)}(X)$ is homologically trivial hence algebraically trivial. The action of correspondences preserves algebraic triviality, and so $C_\ast(b)$ is algebraically trivial.
   \end{proof}
   
   \begin{remark} Nori has conjectured \cite{Nor} that for any smooth projective variety $X$, one should have
   \[ A^2_{AJ}(X)\ \stackrel{??}{\subset}\ A^2_{alg}(X)\ .\]
   More generally, Jannsen has conjectured \cite{J3} that
   \[ F^i A^i(X) \ \stackrel{??}{\subset}\ A^i_{alg}(X)\ \ \ \forall i> 0\ ,\]  
   where $F^\ast A^\ast$ denotes the conjectural Bloch--Beilinson filtration.
    
   Corollary \ref{cor5} fits in nicely with these conjectures; indeed, it is expected \cite{SV} that for varieties $X$ with an MCK decomposition one has equalities
   \[ A^i_{(i)}(X)\stackrel{??}{=} F^i A^i(X)\ \ \ \forall i,\]
   and in particular  
   \[ A^2_{AJ}(X)\stackrel{??}{=}   A^2_{AJ}(X)\ .\]
   \end{remark}
   
  Before stating the next corollary, we first recall a definition:
  
   \begin{definition}[Voevodsky \cite{Voe}]\label{sm} Let $Y$ be a smooth projective variety. A cycle $a\in A^j(Y)$ is called {\em smash--nilpotent\/} 
if there exists $m\in\NN$ such that
  \[ \begin{array}[c]{ccc}  a^m:= &\undermat{(m\hbox{ times})}{a\times\cdots\times a}&=0\ \ \hbox{in}\  A^{mj}(Y\times\cdots\times Y)_{}\ .
  \end{array}\]
  \vskip0.6cm

Two cycles $a,a^\prime$ are called {\em smash--equivalent\/} if their difference $a-a^\prime$ is smash--nilpotent. 
We will write 
    \[A^j_\otimes(Y)\subset A^j(Y)\] 
  for the subgroup of smash--nilpotent cycles.
\end{definition}

\begin{conjecture}[Voevodsky \cite{Voe}]\label{voe} Let $Y$ be a smooth projective variety. Then
  \[  A^j_{hom}(Y)\ \subset\ A^j_\otimes(Y)\ \ \ \hbox{for\ all\ }j\ .\]
  \end{conjecture}

\begin{remark} It is known that $A^j_{alg}(Y)\subset A^j_\otimes(Y)$ for any $Y$ \cite{Voe}, \cite{V9}. In particular, Conjecture \ref{voe} is true for divisors ($j=1$) and for $0$--cycles ($j=\dim Y$).

It is known \cite[Th\'eor\`eme 3.33]{An} that Conjecture \ref{voe} implies (and is strictly stronger than) Kimura's finite--dimensionality conjecture. For partial results concerning Conjecture \ref{voe}, cf. \cite{KS}, \cite{Seb2}, \cite{Seb}, \cite[Theorem 3.17]{V3}, \cite{moismash}.
\end{remark}
   
 Using our hard Lefschetz result, we can prove certain pieces $A^i_{(j)}$ are smash--nilpotent:  
 
 \begin{corollary}\label{cor6} Let $X$ be an abelian variety, or the Hilbert scheme $A^{[m]}$ of an abelian surface, or a generalized Kummer variety. Then
   \[  A^{i+1}_{(i)}(X)\ \subset\ A^i_{\otimes}(X)\ \ \ \forall i>0\ .\]
   \end{corollary}
   
   \begin{proof} Let $n:=\dim X$. We first observe that Voevodsky's conjecture holds for $1$--cycles on $X$, i.e.
   \begin{equation}\label{num} A^{n-1}_{hom}(X) = A^{n-1}_\otimes(X)\ .\end{equation}
   For abelian varieties, this is proven by Sebastian \cite{Seb}. For Hilbert schemes of abelian surfaces and for generalized Kummer varieties, one can apply the argument of \cite[Theorem 3.17]{V3} to prove equality (\ref{num}).
   
   The proof of Theorem \ref{main0} implies there exists a correspondence $C$ such that
   \[ A^{i+1}_{(i)}(X)\ \xrightarrow{ \cdot h^{n-2-i}}\ A^{n-1}_{(i)}(X)\ \xrightarrow{C_\ast}\ A^{i+1}_{(i)}(X) \]
   is the identity. But for $i>0$, any cycle $b\in A^{n-1}_{(i)}(X)$ is homologically trivial, hence (using equality \ref{num})) smash--nilpotent. The action of correspondences preserves smash--nilpotence (since it is known smash--equivalence is an adequate equivalence relation), and so $C_\ast(b)$ is smash--nilpotent. This proves the corollary.
   \end{proof}

  \section{Other cases}
  
  \begin{remark} Some other cases to which Theorem \ref{main0} applies are: the Hilbert square $A^{[2]}$ of an abelian variety $A$ (this has an MCK decomposition thanks to \cite{SV}), the Hilbert cube $A^{[3]}$ of an abelian variety $A$ (this has an MCK decomposition thanks to \cite[Theorem 1]{SV2}), the nested Hilbert schemes $A^{[1,2]}$ and $A^{[2,3]}$ of an abelian variety $A$ (these have an MCK decomposition thanks to \cite[Theorem 6.1]{SV2}; the other hypotheses are satisfied thanks to the explicit description given in \cite[Section 6]{SV2}), the Hilbert scheme $S^{[m]}$ where $S$ is a $K3$ surface with finite--dimensional motive (this has an MCK decomposition thanks to \cite{V6}; the other hypotheses are satisfied thanks to \cite{CM}). 
  
(In addition, let $X=K_m(A)$ be a generalized Kummer variety. {\em If\/} it can be proven that the MCK decomposition of $X$ obtained in \cite{FTV} is self--dual, then it follows from \cite{SV2} that the Hilbert schemes $X^{[2]}$ and $X^{[3]}$ and $X^{[1,2]}$ and $X^{[2,3]}$ also have an MCK decomposition. In this case, these varieties can be added to the above list of cases to which Theorem \ref{main0} applies.)
   \end{remark}

\vskip0.6cm

\begin{acknowledgements} I am grateful to Charles Vial for helpful email discussions, and to the referee for pertinent remarks that helped improve the presentation.
Thanks to Len, Kai and Yasuyo for successfully turning my mind away from work as soon as I leave my office. 
 \end{acknowledgements}

\end{document}